\documentclass[12pt]{lms}
\usepackage{mathtools}
\usepackage{amssymb}
\usepackage{graphicx}
\usepackage{capt-of}
\usepackage{multicol}

\newtheorem{thm}{Theorem}[section]
\newtheorem{prop}[thm]{Proposition}

\newtheorem{lemma}[thm]{Lemma}
\newtheorem{conj}[thm]{Conjecture} 

\numberwithin{equation}{section}

\newcommand{\R}{\mathbb{R}} 
\newcommand{\F}{\mathbb{F}}  
  
\newcommand{\Z}{\mathbb{Z}}
\newcommand{\N}{\mathbb{N}}

\newcommand{\set}[2]{\{#1 \: : \: #2\}}
\newcommand{\lv}{\lvert}
\newcommand{\rv}{\rvert}
\newcommand{\divides}{\: | \:}
\newcommand{\fqstar}{\F_q^*}
\newcommand{\fqx}{\F_q [x]}
\newcommand{\flr}[1]{\lfloor #1 \rfloor}
\newcommand{\angl}[1]{\langle #1 \rangle}
\newcommand{\gc}{\textnormal{gcd }}
\newcommand{\lc}{\textnormal{lcm }}

\title{Roots of Sparse Polynomials over a Finite Field}
\author{Zander Kelley}
\begin{document}
\maketitle

\begin{abstract}
For a $t$-nomial $f(x) =  \sum_{i = 1}^t c_i x^{a_i} \in \mathbb{F}_q[x]$, we show that the number of distinct, nonzero roots of $f$ is
bounded above by $2 (q-1)^{1-\varepsilon} C^\varepsilon$, where $\varepsilon = 1/(t-1)$ and $C$ is the size of the largest coset  in $\mathbb{F}_q^*$ on which $f$ vanishes completely. Additionally, we describe a number-theoretic parameter depending only on $q$ and the exponents $a_i$ which provides a general and easily-computable upper bound for $C$.
We thus obtain a strict improvement over an earlier bound of Canetti et al.\ which is related to the uniformity of the Diffie-Hellman distribution. 
Finally, we conjecture that $t$-nomials over prime fields have only $O(t \log p)$ roots in $\mathbb{F}_p^*$ when $C = 1$.
\end{abstract}

\section{Introduction} 

Over the real numbers, the classical Descartes' Rule implies that the number of distinct, real roots of a $t$-nomial 
$f(x) = c_1 x^{a_1}  + \cdots + c_t x^{a_t} \in \R[x]$ is less than $2t$, regardless of its degree.
It is a natural algebraic problem to look for analogous sparsity-dependent bounds over other fields that are not algebraically closed.
In \cite{canetti}, Canetti et al. derive the following analogue of Descartes' Rule for polynomials in $\fqx$.

\begin{thm}(\cite{canetti}, Lemma 7) \label{D}
For $f(x) = c_1 x^{a_1}  + c_2 x^{a_2} + \cdots + c_t x^{a_t} \in \fqx$ (with $c_i$ nonzero), if $R(f)$ denotes the number of distinct, nonzero roots of $f$ in $\F_q$, then
$$R(f) \leq 2 (q-1)^{1-1/(t-1)} D^{1/(t-1)} + O\left( (q-1)^{1-2/(t-1)} D^{2/(t-1)} \right) ,$$
where
$$D(f) = \min_{i} \max_{j \neq i} \: \left( \gcd(a_i - a_j, q-1) \right) .$$
\end{thm}

For $\vartheta \in \F_p^*$, the associated Diffie-Hellman distribution is defined by the random variable $(\vartheta^x, \vartheta^y, \vartheta^{xy})$ where $x$ and $y$  are uniformly random over $\{1,\ldots ,p-1\}$.
The Diffie-Hellman cryptosystem relies on the assumption that an attacker cannot easily determine $\vartheta^{xy}$ given the values of $\vartheta$, $\vartheta^x$, and $\vartheta^y$.
In \cite{canetti}, Canetti et al. showed that Diffie-Hellman distributions are very nearly uniform (which is an important property for the security of the cryptosystem), and the bound in Theorem \ref{D} was the central tool which powered their arguments.

Since then, the bound has been a useful tool for studying various algorithmic and number-theoretic problems:
 in \cite{approx} it was used to study the complexity of recovering a sparse polynomial from a small number of approximate values (which is relevant to the security of polynomial pseudorandom number generators),
 in \cite{vand} it was used to study the singularity of generalized Vandermonde matrices over $\F_q$,
 in \cite{exp} it was used to study the solutions of exponential congruences $x^x = a \bmod p$,
and  in \cite{binary} it was used to study the correlation of linear recurring sequences over $\F_2$.
The main result of this paper is a new bound (Theorem \ref{bound} of Section 2 below) improving
   Theorem \ref{D} by removing the asymptotic term and replacing $D$
   by a smaller, intrinsic parameter.

\section{Statement of Results} 

More recently in \cite{rojas1}, Bi, Cheng, and Rojas studied the computational complexity of deciding whether a $t$-nomial $f$ has a root in $\fqstar$.
They gave a sub-linear algorithm for this problem; we give a rough sketch here. 
First, they efficiently replace any instance with a $t$-nomial of degree bounded by $2(q-1)^{1-1/(t-1)}$ (while preserving the answer to
the decision problem), and then they compute the greatest common divisor of this instance and $x^{q-1} - 1$, which can be done in time proportional to the degree of the instance.
As a result of their investigation, they derive the following characterization of the roots of a sparse polynomial in $\fqx$.

\begin{thm} (\cite{rojas1}, Theorem 1.1) \label{delta}
For $f(x) = c_1 x^{a_1}  + c_2 x^{a_2} + \cdots + c_t x^{a_t} \in \fqx$, define $\delta(f) = \gcd(a_1,a_2,\ldots,a_t,q-1)$.
The set of nonzero roots of $f$ in $\F_q$ is the union of no more than
$$ 2 \left( \frac{q-1}{\delta} \right)^{1-1/(t-1)} $$
cosets of two subgroups $H_1 \subseteq H_2$ of $\fqstar$, where
\begin{align*}
\lv H_1 \rv = \delta \;\; \textnormal{   and   } \;\;
\lv H_2 \rv \geq \delta^{1-1/(t-1)} (q-1)^{1/(t-1)}.
\end{align*}

\end{thm}

This result does not immediately yield any bound on the number of roots $R(f)$ since there is no upper bound given for the size of the $H_2$-cosets.
However, if for some reason we were assured that the set of roots was a union of only $H_1$-cosets, we could conclude
$$ R(f) \leq  \delta \cdot 2 \left( \frac{q-1}{\delta} \right)^{1-1/(t-1)} = 2 (q-1)^{1-1/(t-1)} \delta^{1/(t-1)} ,$$
which is an improvement on Theorem \ref{D} since it can be easily checked that $\delta(f) \leq D(f)$ always.

\begin{thm} \label{coset}
For $f(x) = c_1 x^{a_1} + \cdots + c_t x^{a_t} \in \F_q[x]$ (with $c_i$ nonzero),  define
\begin{align*}
S(f) &:= \set{k \divides (q-1)}{\textnormal{for all } i, \textnormal{ there is a } j \neq i \text{ with } a_i \equiv a_j \bmod k}.
\end{align*} 
If $f$ vanishes completely on a coset of size $k$, then $k \in S(f)$.
\end{thm}

\begin{proof}
For some generator $g$ of $\fqstar$, let $\alpha \langle g^{\frac{q-1}{k}} \rangle$ denote a coset of the unique subgroup of order $k$ in $\fqstar$, and let $\beta = \alpha^k$.
The members of this coset are exactly the roots of the binomial $x^k - \beta$. 
So, $f$ vanishes completely on this coset if and only if $(x^k - \beta) \divides f$, or equivalently if  $f \equiv 0 \bmod (x^k - \beta)$.

To see when this happens, we view $f$ in the ring $\fqx / \langle x^k - \beta \rangle$. 
In this ring, we have the relation $x^k \equiv \beta$, so if each $a_i$ has remainder $r_i \bmod k$, then
$$  f \equiv c_1 \beta^{{\flr{a_1/k}}} x^{r_1} + \cdots + c_t \beta^{{\flr{a_t/k}}} x^{r_t} \mod (x^k - \beta). $$
Now $f$ might be identically zero (in this ring) since the $r_i$'s are not necessarily distinct. 
However, there is one obvious barrier to this: if just one $r_i$ is unique, then $f$ in particular contains the nonzero monomial $(c_i \beta^{{\flr{a_i/k}}}) x^{r_i}$. $f \equiv 0$ requires that each remainder $r_i$ has at least one ``partner" $r_j = r_i$ so that monomials can cancel. Therefore $(x^k -\beta) \divides f$ implies that, for each $i \in \{1,2,\ldots,t\}$, there is some $j \neq i$ with $a_i \equiv a_j \bmod k$.
\end{proof}

Thus $S(f)$ lists the sizes of cosets on which $f$ might possibly vanish completely.
For example, if $a_1 = 0$ and the other exponents $a_{i>1}$ are all prime to $q-1$ then $S(f) = \{1\}$, and so it is structurally impossible for $f$ to vanish completely on any nontrivial coset, regardless of choice of coefficients $c_i \in \fqstar$.
On the other hand whenever $k \in S(f)$, there is some choice of $c_i \in \fqstar$ so that $f$ does indeed vanish completely on a given coset of size $k$.

When $\textnormal{max}(S) < \delta^{1-1/(t-1)} (q-1)^{1/(t-1)}$, Theorem \ref{coset} can be combined with Theorem \ref{delta} to get a bound on $R(f)$ by ruling out the possibility of $H_2$-cosets. If $ \textnormal{max}(S)$ is any larger, Theorem \ref{delta} is no longer helpful; the most we can conclude is that
$$ R(f) \leq \lv H_2 \rv \cdot 2 \left( \frac{q-1}{\delta} \right)^{1-1/(t-1)} \leq \textnormal{max}(S) \cdot 2 \left( \frac{q-1}{\delta} \right)^{1-1/(t-1)} ,$$
which is worse than trivial: $R(f) < q - 1$. However, $S(f)$ turns out also to be independently useful for deriving sparsity-dependent bounds.

\begin{thm} \label{bound}
Let $f(x) = c_1 x^{a_1} + \cdots + c_t x^{a_t} \in \F_q[x]$ (with $c_i$ nonzero), and let $\delta(f)$ be defined as above, and let $C(f)$ denote the size of the largest coset in $\fqstar$ on which $f$ vanishes completely.
If $R(f)$ denotes the number of distinct, nonzero roots of $f$ in $\fqstar$, then we have
$$ R(f) \leq 2 (q-1)^{1-1/(t-1)} C^{1/(t-1)}, $$
and furthermore if $C < \delta^{1-1/(t-1)} (q-1)^{1/(t-1)} $, then
$$ R(f) \leq  2 (q-1)^{1-1/(t-1)} \delta^{1/(t-1)} .$$

\end{thm}

This result is a strict improvement on Theorem \ref{D}, since, as we will see, $D(f)$ is in particular an upper bound for $S(f)$ and therefore also for $C(f)$.
In fact, we can get another easily computable upper bound for $S(f)$ that is in general tighter than $D(f)$.

\begin{prop} \label{params}
For $f(x) = c_1 x^{a_1} + \cdots + c_t x^{a_t} \in \F_q[x]$ define the parameters
\begin{align*}
\delta(f) &=\gcd(a_1,a_2, \ldots,a_t, q-1) \\
D(f) &= \min_{i} \max_{j \neq i} \: \left( \gcd(a_i - a_j, q-1) \right)  \\
Q(f) &= \underset{i}{\gc} \underset{j \neq i}{\lc}(\gcd(a_i - a_j, q-1)) \\
K(f) &= \min_{i} \max_{j \neq i} \: \left( \gcd(a_i - a_j, Q)\right)  
\end{align*}
These parameters relate to $S(f)$ as follows.
\begin{itemize}
\item $\delta(f) \in S(f)$.
\item For all $k \in S(f)$, $k \divides Q(f)$.
\item $D(f)$, $Q(f)$, and $K(f)$ are all upper bounds for $S(f)$, and $K(f) \leq \min(D(f), Q(f))$.
\end{itemize}

\end{prop}

\section{Optimality of the Bound} 

Since the polynomial which defines a given function on $\F_q^*$ is unique only up to equivalence $\mod x^{q-1} - 1$, we restrict our attention to polynomials
with degree less than $q-1$. Thus we fix the following notation. Recall that $C(f) \leq 1$ indicates that $f$ does not vanish on any entire coset of any nontrivial subgroup of $\F_q^*$.

    \begin{multicols}{2}
    \begin{itemize}
        \item  $\mathcal{F}(q) = \set{f \in \F_q[x]}{\deg f < q - 1}$
        \item $\mathcal{F}(q,t) = \set{f \in \mathcal{F}(q)}{ f \textnormal{ has } t \textnormal{ terms}}$
        \item  $\mathcal{F}(q,t,r) = \set{f \in \mathcal{F}(q,t)}{R(f) = r}$
        \item $\mathcal{F}_1(q) = \set{f \in \mathcal{F}(q)}{C(f) \leq 1}$
        \item $\mathcal{F}_1(q,t) = \set{f \in \mathcal{F}_1(q)}{ f \textnormal{ has } t \textnormal{ terms}}$
        \item  $\mathcal{F}_1(q,t,r) = \set{f \in \mathcal{F}_1(q,t)}{R(f) = r}$
    \end{itemize}
    \end{multicols}

In this section, we consider the possibility that the bound in Theorem \ref{bound} can be improved.
Consider the binomial $f(x) = x^{(q-1)/2} + 1$. When $q$ is odd, this binomial vanishes at every non-square in $\F_q^*$, and consequently $R(f) = C(f) = (q-1)/2$.
More generally, when $(q-1)$ is divisible by $t$, the $t$-nomial $f(x) = (x^{q-1} - 1)/(x^{(q-1)/t} - 1)$ vanishes on $t-1$ cosets of size $(q-1)/t$, and so $R(f) = (q-1)(1-1/t)$.
These examples show that there is no hope to improve the bound in Theorem \ref{bound} by removing the dependence on $C(f)$. However, the proportion of polynomials
with $C(f) > 1$ is small, so it may be worthwhile to search for improved bounds for $f \in \mathcal{F}_1(q)$.

\begin{prop} \label{c2}
$$ \frac{ |\mathcal{F}(q) \setminus \mathcal{F}_1(q) | }{|\mathcal{F}(q)|} = O \left(\frac{1}{q} \right). $$
\end{prop} 
\begin{proof}

If $f \in \mathcal{F}(q)$ vanishes on a nontrivial coset, then it vanishes on a coset of prime order. Thus we can bound the number of such $f  \in \mathcal{F}(q)$ by counting polynomials of the form
$$ \left(x^p - \beta \right) \sum_{j = 0}^{q - 2 - p} c_j x^j, $$
where $p$ divides $q-1$ and $\beta$ lies in the subgroup of $\F_q^*$ of size $(q-1)/p$. Thus the proportion of $f$ with $C(f) > 1$ is bounded by
$$ \frac{1}{|\mathcal{F}(q)|} \sum_{p \: | \: q - 1} \left( \frac{q-1}{p} \right)  q^{q-1-p} 
\leq  \sum_{p \: | \: q - 1}  q^{1-p} \leq   q^{-1} +   \sum_{\substack{p \: | \: q - 1 \\ p \; > \: 2}}  q^{1-p} = \frac{1}{q} + O \left( \frac{\log q }{q^2} \right).$$
Note that we have used the well-known fact that the number of distinct prime factors of an integer $n$ is bounded by $\log n$.
\end{proof}

In \cite{rojas2}, the authors investigate the existence of sparse polynomials with many roots. When $q$ is a $t$-th power, they give the $t$-nomial
$f(x) = 1 + \sum_{i = 1}^{t-1} x^{(q^{i/t} - 1)/(q^{1/t} - 1)} \in \F_q[x]$, which has $R(f) \geq  q^{1-2/t}$. Furthermore, when $t$ is prime they show that $D(f) \leq t/2$.
In the case when $q$ is an odd square, the authors of \cite{tri} give the trinomial $f(x) = x^{q^{1/2}} + x - 2$ which has $R(f) = q^{1/2}$, and it is shown that $C(f) = 1$.
Thus, both examples are able to attain a large number of roots in $\F_q^*$ without vanishing on a large coset, and they show that the $O(q^{1-1/(t-1)})$ bound from Theorem \ref{bound}
is nearly optimal in the general setting. However, these examples both share a special property - they vanish on entire translations of a subspace of $\F_q$. We are unaware
of any example of a sparse polynomial which has a large number of roots in ``general position." Consequently, we propose that a much better bound is possible
for the special case of prime fields $\F_p$, which have no proper subfields.

Let $R_{p,t} = \max \set{R(f)}{f \in \mathcal{F}_1(p,t)}$. Obviously $R_{p,1} = 0$ and $R_{p,2} = 1$, because monomials have no roots in $\F_p^*$, and a binomial
defines a coset in $\F_p^*$ if it has a root at all. We have checked by computer that the following inequalities hold.

\begin{itemize}
\item $R_{p,3} <  1.8 \log p$ for $p \leq 139571$
\item $R_{p,4} <  2.5 \log p$ for  $p \leq 907$
\item $R_{p,5} <  2.9 \log p$ for $p \leq 101$
\end{itemize}
Therefore, the current bound of $R_{p,t} = O \left( p^{1-1/(t-1)} \right)$ appears to be far from optimal for $t$-nomials over $\F_p$ which do not vanish on a nontrivial coset.

It is easy to see that the proportion of polynomials $f \in \mathcal{F}(p)$ which have $R(f) = r$ is bounded by $1/r!$.
Indeed, simply count the proportion of polynomials of the form
$$ \left( \prod_{i=1}^r (x - \alpha_i) \right) \left( \sum_{i=0}^{p-2-r} c_i x^i \right), $$
with $\alpha_i \in \F_p^*$ distinct, which gives
$$  \frac{\binom{p - 1}{r} p^{p-1-r}}{| \mathcal{F}(p)|} =  \binom{p - 1}{r} \frac{1}{p^r}  \leq \frac{1}{r!}.$$

With this in mind, we propose that the observed logarithmic behavior of $R_{p,t}$ can be explained by the following heuristic. Let $t(f)$ denote the number of nonzero terms of $f$.

\vspace{0.2cm}
\centerline{  \textit{
\vspace{0.2cm}
$R(f)$ and $t(f)$ are statistically independent properties of a random $f \in \mathcal{F}_1(p)$.} }

This heuristic does not hold precisely, but it motivates the following conjecture, which captures the sentiment while allowing for some error.

\begin{conj} \label{con}
There exists a constant $\gamma > 0$ such that

$$ \frac{ | \mathcal{F}_1(p,t,r) |}{|\mathcal{F}_1(p,t) |}\leq \left( \frac{1}{r!} \right)^\gamma$$

for all $p$ prime, $t \in \N$, and $r \in \N$.
\end{conj}

We have checked by computer that the inequality in Conjecture \ref{con} holds with $\gamma = 1/2$ for all $r \in \N$ in the following cases.

\begin{itemize}
\item $t = 3$, $p \leq 30977$
\item $t = 4$, $p \leq 907$
\item $t = 5$, $p \leq 101$
\end{itemize}

\begin{thm}
If Conjecture \ref{con} is true, then $R_{p,t} = O \left(t \log p \right)$.
\end{thm}
\begin{proof}
Suppose Conjecture \ref{con} is true. Then we have
$$ | \mathcal{F}_1(p,t,r) | \leq |\mathcal{F}_1(p,t) | \left( \frac{1}{r!} \right)^\gamma \leq p^{2t} / (r!)^\gamma. $$
If $p^{2t} / (r!)^\gamma < 1$ then the set $\mathcal{F}_1(p,t,r) $ is empty, so we must have $p^{2t} / (R_{p,t}!)^\gamma \geq 1$, or equivalently,
$\log(R_{p,t}!) \leq \log(p^{2t/\gamma})$. Applying Stirling's approximation, we get
$$ R_{p,t} \leq R_{p,t} \log R_{p,t} \sim \log(R_{p,t}!) \leq (2/\gamma) t \log p = O\left( t \log p\right).$$
\end{proof}

For a more detailed account of the computational and heuristic support for the conjectural logarithmic bound in the case of trinomials, see \cite{rojas2} and \cite{tri}.

\section{Proofs} 

The general strategy employed here (and in both \cite{rojas1} and \cite{canetti}) for obtaining sparsity-dependent bounds on $R(f)$ can be loosely sketched as follows.
Consider integers $e$ prime to $q-1$, which have the property that the map $x \mapsto x^e$ is a bijection on $\fqstar$.
Since  $x \mapsto x^e$ simply permutes the elements of $\F_q^*$, we have $R(f(x)) = R(f(x^e))$. 
Furthermore, $f(x^e)$ is equivalent (as a mapping on $\fqstar$) to any $g(x) = c_1 x^{b_1} + \cdots + c_t x^{b_t}$ with $b_i\equiv ea_i \bmod (q-1)$.
Thus the basic idea is to find some $e$ so that the remainders of $ea_i \bmod (q-1)$ are all small, yielding a $g$ of small degree, and so $R(f) = R(g) \leq deg(g)$.

The following lemma, a fact about the geometry of numbers, will be our main tool for achieving the desired degree reduction. 

\begin{lemma} \label{packing}
Fix the natural numbers $a_1 , a_2 ,\ldots ,a_t , N$. 
If $n \leq N/ \gcd(a_1,a_2,\ldots,a_t,N)$, there is an $e \in \{1,2,\ldots,n-1\}$ and a $v \in N \Z^t$ so that 
$$0 < \max_{1 \leq i \leq t} \; \lv ea_i + v_i \rv \leq N/n^{1/t}.$$

\end{lemma}

\begin{proof}
Consider the vectors $l_i = i(a_1,\ldots,a_t) = (ia_1,\ldots,ia_t) \in (\R / N \Z)^t$ for $i \in \{1,2,\ldots n\}$.
Let $\|\cdot\|_\infty$ denote the standard infinity norm on $\R^t$.
We wish to view these vectors geometrically as points in $\R^t$, but they are only defined up to equivalence in $(\R / N \Z)^t$, so define
$$ \| l \|_N  = \min_{v \in N\Z^t} \|l + v\|_\infty,$$
which gives the smallest norm of any representative of the equivalence class $l + N\Z^t$ viewed as a point in $\R^t$
(equivalently, $\| l \|_N$ gives the distance from $l$ to the nearest lattice point in $N\Z^t$).
Suppose that
$$d = \min_{i \neq j} \|l_j - l_i \|_N .$$
Since the vectors are all at least $d$ apart, the sets 
$$B_i = \set{x \in (\R / N \Z)^t}{\|x - l_i \|_N < d/2} $$ 
are disjoint, so each $l_i$ sits in its own personal box of volume $d^t$. 
We may choose to represent these $n$ disjoint sets uniquely in the fundamental domain $[0,N)^t$, which has volume $N^t$. 
Therefore we have a total volume of $n \cdot d^t$ sitting in a volume of $N^t$; we conclude that $d \leq N/n^{1/t}$.

Note that the modular definition of distance is crucial here; consider instead $n$ points in $[0,N]^t$ that are $d$-separated only in the standard $l_\infty$ metric.
A volume-packing argument becomes more complicated in this case because the box around a point near the boundary may lie partly outside $[0,N]^t$ (it doesn't ``wrap around"), and so some points do not absorb a full $d^t$ worth of volume from $[0,N]^t$.

To finish, we find $i,j$ (with $1 \leq i<j \leq n$) so that $\|l_j - l_i \|_N = d$ and set
$ l_e = l_{(j-i)}= (j-i)(a_1,\ldots,a_t) = l_j - l_i .$
We have  
$$\| l_e \|_N = \min_{v \in N\Z^t} \|(ea_1,\ldots,ea_t) + v\|_\infty \leq N/n^{1/t}, $$
and $e$ satisfies $1 \leq e \leq n - 1$.
The subgroup of $(\Z/N\Z)^t$ generated by $(a_1, \ldots, a_n)$ has order $N/ \gcd(a_1,\ldots,a_t,N) \geq n$. Since $0 < e < n$, 
$e(a_1, \ldots, a_n)  \not\equiv (0,\ldots,0) \in (\Z/N\Z)^t$, which verifies that $\| l_e \|_N > 0$.

\end{proof}

Lemma \ref{packing} and its proof are extremely similar in spirit to the argument used by Canetti et al. in \cite{canetti}. 
They also viewed the $n$ vectors as points in $[0,N)^t$, but to find a pair of nearby points they partitioned the hypercube into $< n$ equally-sized sub-cubes and appealed to the pigeonhole principle.
Here we were able to avoid this discretization of space which lead to the small asymptotic term appearing in Theorem \ref{D}, which turns out to be unnecessary.

\begin{proof}[of Theorem \ref{bound}]
The second claim is immediate from Theorem \ref{delta}, since there can be no $H_2$-cosets of roots.
We now prove the first claim.

Let $f(x) = c_1 x^{a_1} + c_2 x^{a_2} + \cdots + c_t x^{a_t} \in \F_q[x]$ with $c_i$ nonzero, and let $C$ denote the size of the largest coset in $\fqstar$ on which $f$ vanishes completely.
For our purposes, we may assume that $a_1 = 0$, since otherwise we can write
\begin{align*}
 f(x) &= x^{a_1} \tilde{f}(x) \\
 \tilde{f}(x) &= c_1 + c_2 x^{a_2 - a_1} + \cdots + c_t x^{a_t - a_1},
\end{align*}
showing that $f$ has a root at zero, but its nonzero roots are just the roots of $\tilde{f}$, so $R(f) = R(\tilde{f})$ and $C(f) = C(\tilde{f})$.
Therefore we continue assuming that $a_1 = 0$.

Consider $\delta(f) = \gcd(a_2, \ldots, a_t, q-1)$. 
The nonzero roots of $f(x) = c_1 + c_2 x^{a_2} + \cdots + c_t x^{a_t}$ are in one-to-one correspondence with the solutions of the system
\begin{align*}
c_1 + c_2 y^{a_2/\delta} + \cdots + c_t y^{a_t/\delta} &= 0 \; &y \in \angl{g^\delta} \\
x^\delta &= y \: &x \in \fqstar
\end{align*}
If $f$ has no roots in $\fqstar$ then our bound is of course true, so suppose this system has at least one solution $(y_0,x_0)$.
Then in fact the system has at least $\delta$ solutions and $f$ vanishes on the coset $\set{x}{x^\delta = y_0}$.
This allows us to conclude that $C \geq \delta$, and so $\left(\frac{q-1}{C}\right) \leq (q-1)/ \gcd(a_2, \ldots, a_t, q-1)$.

Therefore we can apply Lemma \ref{packing} to find an $e \in \{1,2,\ldots,\frac{q-1}{C}-1\}$ and a $v \in (q-1) \Z^{t-1}$ so that 
$$0 < \| (ea_2, \ldots, ea_t) + v\|_\infty \leq \left(q-1 \right)/\left( \frac{q-1}{C} \right)^{1/(t-1)} .$$
Suppose $k = \gcd(e,q-1) = 1$.
Then the mapping $x \mapsto x^e$ is a bijection on $\fqstar$ that simply re-orders the elements of $\F_q^*$, thus $R(f(x)) = R(f(x^e))$.
We are interested in $f(x^e)$ only as a function on $\fqstar$ (rather than as a formal object in $\fqx$), and since $\fqstar$ is a group of order $q-1$, this function is not changed by shifting its exponents by $v_i \in (q-1)\Z$. Thus we may represent the function $f(x^e)$ as the (possibly Laurent) polynomial 
$$ f(x^e) = c_1 + c_2 x^{ea_2 + v_2} + \cdots + c_t x^{ea_t + v_t} ,$$
which satisfies 
$$0 < M = \max_{1 \leq i \leq t} \lv ea_i + v_i \rv \leq (q-1)^{1-1/(t-1)} C^{1/(t-1)}. $$
Again we are only interested in nonzero roots; note that $R(f(x^e)) = R(x^M f(x^e))$.
Since $x^M f(x^e)$ is an honest polynomial in $\fqx$ with non-negative exponents, we have $R(f) = R(x^M f(x^e)) \leq \deg(x^M f(x^e)) \leq 2M$ and we are done.

However, we might have $k = \gcd(e, q-1) > 1$. 
In this case $x \mapsto x^e$ is not a bijection - it takes $\fqstar = \angl{g}$ to a smaller subgroup $\angl{g^e} = \angl{g^k}$ of size $\left( \frac{q-1}{k} \right)$. 
However, we can still cover $\fqstar$ by $k$ cosets of this subgroup.
We have
$$ R(f(x^e)) = \sum_{i = 0}^{k-1} \frac{1}{k} R(f(g^i x^e)) ,$$
since $\fqstar = \bigcup_{i = 0}^{k-1} g^i \angl{g^e}$, and $x^e = y$ has $k$ solutions for each $y \in \angl{g^e}$.
Now we repeat our earlier tricks and arrive at 
$$ R(f) \leq \sum_{i = 0}^{k-1} \frac{1}{k} \deg (x^M f(g^i x^e)) \leq 2M ,$$
except that we must be careful that no $f(g^i x^e)$ is identically zero, preventing us from using degree to bound root number.
If  $f(g^i x^e)$ is identically zero then $f$ vanishes completely on the coset $g^i \angl{g^e} = g^i \angl{g^k}$ of size $\left(\frac{q-1}{k}\right)$.
However, since $k = \gcd (e,q-1) \leq e < \left( \frac{q-1}{C} \right)$, we have
$$\frac{q-1}{k} > \frac{q-1}{\left(\frac{q-1}{C}\right)} = C,$$ 
so this is impossible by the definition of $C$; the cosets are too large for $f$ to vanish on completely.

\end{proof}

\begin{proof}[of Proposition \ref{params}]
For $f(x) = c_1 x^{a_1} + \cdots + c_t x^{a_t} \in \fqx$, we have the following equivalent definitions for $S$:
\begin{align*}
S(f) &= \set{k \divides (q-1)}{\forall i, \exists j \neq i \text{ such that } a_i \equiv a_j \bmod k}  \\
&= \set{k \divides (q-1)}{\forall i, \exists j \neq i \text{ such that } k \divides (a_i - a_j)}\\
&= \set{k \in \N}{\forall i, \exists j \neq i \text{ such that } k \divides \gcd(a_i - a_j, \: q-1)} \\
&= \bigcap_{i = 1}^{t} \bigcup_{j \neq i} \set{k \in \N}{k \divides \gcd(a_i - a_j, \: q-1) }.
\end{align*}
Clearly by the second definition we have $\delta(f) = \gcd (a_1,a_2, \ldots, a_t, q-1) \in S$.
From the fourth definition we can get an upper bound for $S$ by passing to the superset
$$ \bigcap_{i = 1}^{t} \bigcup_{j \neq i} \set{k \in \N}{k \leq \gcd(a_i - a_j, \: q-1) } \supseteq S ,$$
which has maximal element
$$D = \min_{i} \max_{j \neq i} \: \left( \gcd(a_i - a_j, q-1)\right) . $$
Alternatively, by considering a different lattice structure on the integers, we can pass to the superset
$$ \bigcap_{i = 1}^{t} \set{k \in \N}{k \divides \gcd(L_j, q-1) )} = \set{k \in \N}{k \divides Q}   \supseteq S ,$$ 
where
\begin{align*}
L_i &= \textnormal{lcm}(a_i - a_1, \ldots, a_i - a_{i-1}, a_i - a_{i+1}, \ldots ,a_i - a_t) ,\\
Q &= \gcd(L_1, \ldots, L_t , q - 1) = \underset{i}{\gc} \underset{j \neq i}{\lc}(\gcd(a_i - a_j, q-1)) .
\end{align*}

Since we now know that, in the end, any member of $S$ must be a divisor of $Q$, we can redefine $S$ (equivalently) using a smaller ambient space:
\begin{align*}
S(f) &= \set{k \divides Q}{\forall i, \exists j \neq i \text{ such that } k \divides (a_i - a_j)}  \\
       &= \bigcap_{i = 1}^{t} \bigcup_{j \neq i} \set{k \in \N}{k \divides \gcd(a_i - a_j, \: Q) } \\
       &\subseteq \bigcap_{i = 1}^{t} \bigcup_{j \neq i} \set{k \in \N}{k \leq \gcd(a_i - a_j, \: Q) }.
\end{align*}
Considering the maximal element of this last superset of $S$ gives the final upper bound
$$ K = \min_{i} \max_{j \neq i} \: \left( \gcd(a_i - a_j, Q) \right) , $$  
which is obviously no larger than either $D$ or $Q$.
\end{proof}


\section*{Acknowledgement}
I would like to thank my advisor, Dr.\ J.M.\ Rojas, for introducing me to this problem, and the Texas A\&M Supercomputing Facility for access to computational resources.

\end{document}